\documentclass[twoside,10pt]{amsart}

\usepackage{amsthm,amsmath,amsfonts,epsfig,amssymb}
\usepackage[T1]{fontenc}
\usepackage{graphicx}
\usepackage{enumerate}
\usepackage{xy}
\usepackage{graphics}
\xyoption{all}
\usepackage{scalefnt}
\usepackage{color}
\usepackage{mathrsfs}
\usepackage{hyperref}
\entrymodifiers={+!!<0pt,\fontdimen22\textfont2>}
\usepackage{amscd}
\usepackage{scalefnt}

\newtheorem{thm}{Theorem}[section]

\newtheorem{propn} [thm]{Proposition}
\newtheorem{lemma} [thm]{Lemma}
\newtheorem{cor}[thm]{Corollary}

\newtheorem{imex}[thm]{Important example}
\newtheorem*{thm*}{Theorem}
\newtheorem{quest}{Question}
\theoremstyle{definition}
\newtheorem{defin}[thm]{Definition}
\theoremstyle{remark}
\newtheorem{rem}[thm]{Remark}
\newtheorem{exem}[thm]{Example}

\def\Z{\mathbb{Z}}
\def\N{\mathbb{N}}

\def\K{\mathbb{K}}

\def\C{\mathcal C}
\def\O{\mathcal O}
\def\T{\mathcal T}

\def\P{\mathcal{P}}
\def\G{\mathcal{G}}
\def\A{\mathcal{A}}

\def\Mfl{\mathcal{M}_{fl}}

\def\X{\mathscr{X}}
\def\L{\mathscr{L}}

\def\p{\mathfrak p}
\def\q{\mathfrak q}
\newcommand{\iso}{\xrightarrow{\sim}}

\newcommand*\xbar[1]{%
  \hbox{%
    \vbox{%
      \hrule height 0.1pt 
      \kern0.5ex
      \hbox{%
        \kern-0.1em
        \ensuremath{#1}%
        \kern-0.1em
      }%
    }%
  }%
} 

\def\bA{\xbar{\A}}
\def\sbA{\mbox{\scriptsize $\bA$}}
\def\FP{\mbox{$\xbar{\P(R)}$}}
\def\sFP{\mbox{\scriptsize $\FP$}}
\DeclareMathOperator{\wid}{Width}
\def\sbe{\subset}
\def\ext{\text{Ext}}
\def\tor{\text{Tor}}
\def\hm{\text{Hom}}
\def\ch{\text{Ch}}
\def\mspec{\text{MaxSpec}}
\def\spec{\text{Spec}}
\def\supp{\text{Supp}}
\def\supph{\text{Supph}}
\def\ker{\text{ker}}

\def\id{\text{Id}}
\def\kos{\text{Kos}}
\def\End{\text{End}}
\def\da{\dagger}
\DeclareMathOperator{\dm}{\text{dim}}
\DeclareMathOperator{\codm}{\text{codim}}
\DeclareMathOperator{\M}{\text{mod}}
\begin{document}

\title{Finite homological dimension and a derived equivalence}
\author{William Sanders and Sarang Sane}
\date{}

\address{William Sanders\\
Department of Mathematics \\
University of Kansas \\
405 Snow Hall \\
1460 Jayhawk Blvd \\ 
Lawrence, Kansas 66045-7594}
\email{wsanders@ku.edu}
\address{Sarang Sane*\\
Department of Mathematics \\
Indian Institute of Technology Madras\\
Sardar Patel Road \\
Chennai India 600036}
\email{sarangsanemath@gmail.com}

\keywords{Derived categories, finite projective dimension, finite length,
resolving subcategories, K-theory, derived Witt groups, Gersten complexes}
\subjclass[2010]{Primary: 13D09; Secondary:18E30, 18G35, 19D99, 19G38}

\begin{abstract}
For a Cohen-Macaulay ring $R$, we exhibit the equivalence of the bounded derived
categories of certain resolving subcategories, which, amongst other results, yields
an equivalence of the bounded derived category of finite length and finite
projective dimension modules with the bounded derived category of projective modules
with finite length homologies. This yields isomorphisms of K-theory and Witt groups (amongst other invariants)
and improves on terms of associated spectral sequences and Gersten complexes.
\end{abstract}

\maketitle


\section{Introduction} 
 \label{sec1}
Let $R$ be a commutative Noetherian ring. Let $\P(R)$ be the category of
finitely generated projective modules, $\FP$ be the category of finite projective
dimension and $\Mfl$ be the category of finite length modules. Let $D^b_{fl}(\P(R))$
be the full subcategory of $D^b(\P(R))$ consisting of bounded complexes with finite
length homologies. 

The category $\Mfl \cap \FP$ has received much attention in recent years.  Bass conjectured that $\Mfl \cap \FP$ contains nonzero modules if and only if the ring is Cohen-Macaulay.  This conjecture was resolved in the affirmative by the famed new intersection theorem of  \cite{R}.  Furthermore, in order to discuss intersection multiplicity over nonregular rings,  Roberts and Srinivas  show in \cite[Proposition 2]{RS} that the Grothendieck groups of $\Mfl \cap \FP$ and $D^b_{fl}(\P(R))$ coincide when $R$ is a Cohen-Macaulay ring.

We prove the following theorem which appears as Theorem \ref{thm-cm-iff}.  
\begin{thm}\label{1.1}
If $R$ is a local ring, then $R$ is Cohen-Macaulay if and only if there is an equivalence of derived categories
$$ D^b(\FP \cap \Mfl) \simeq D^b_{fl}(\P(R)).$$
\end{thm}

When $R$ is Cohen-Macaulay, the equivalence is a special case of our main result, Theorem \ref{main}.  The main difficulty in proving the equivalence lies in showing that the natural functor $ D^b(\FP \cap \Mfl)\to D^b_{fl}(\FP)$ is essentially surjective and full, where $ D^b_{fl}(\FP)$ is the thick subcategory of $D^b(\FP)$ consisting of complexes with finite
length homologies. We overcome this difficulty using a subtle induction on the number of homology modules, an argument which dates back to the preprint \cite{F} (refer to Theorem \ref{kos-lemma2} and Theorem \ref{kos-thm} for our versions).

When $R$ is not Cohen-Macaulay, the equivalence fails because the new intersection theorem \cite{R} asserts that such a ring $R$ never admits a finite length, finite projective dimension module.
Thus $D^b(\FP \cap \Mfl) = 0$. However, the Hopkins-Neeman theorem  \cite{H,N2} states that the thick subcategories of $D^b(\P(R))$ are in bijective correspondence with specialization closed subsets of $\spec(R)$, hence $D^b_{fl}(\FP)$ cannot be $0$. See the proof of Theorem \ref{thm-cm-iff} for more details.

Using Theorem \ref{1.1}, when $R$ is Cohen-Macaulay, we can conclude that the non-connective
$\K$-theory spectra of $\Mfl$ and $D^b_{fl}(\P(R))$  are homotopy equivalent. Special cases of this result for the connective
$K$-theory spectrum can be obtained from \cite[Ex. 5.7]{TT}, and similar results comparing the $K_0$
groups are in \cite[Proposition 2]{RS} (as previously mentioned) and \cite{FH}.
The homotopy equivalence of the $\K$-theory spectra is obtained 
by putting together two equivalences in Theorem
\ref{thick-main} and Lemma \ref{resolv-main} , both of which are induced through natural functors
of the chain complex categories (a zigzag of equivalences induced from chain complex functors).

The equivalence of $K$ groups allows us to refine the terms in weak Gersten complexes for certain non-regular schemes. Let $X$ be a Noetherian scheme, $c$ a non-negative integer, $Coh(X)$ the
category of coherent $\O_X$-modules, $Coh(X)^c$ the subcategory of modules with
codimension at least $c$, and $D^c(Coh(X))$ the subcategory of $D^b(Coh(X))$
consisting of complexes with homologies in $Coh(X)^c$.  Using the natural coniveau
filtration by codimension, one obtains the classical Brown-Gersten-Quillen spectral
sequences of $\K$-groups which abut to the $\K$-theory of $Coh(X)$. Applying Quillen's
localization and d\'{e}vissage theorems, the terms occurring in these sequences can be
identified with $\K$-groups of the residue fields of the points. Classically, these
spectral sequences were applied in the case when $X$ was Noetherian, regular and
separated, in which case they converged to the $\K$-groups of $X$ \cite{Q} . This is
because for such $X$ there is a well-known equivalence 
$D^c(VB_X) \stackrel{\xi}{\rightarrow} D^c(Coh(X))$
(since there is an ample family of line bundles) which yields
$D^b(Coh(X)^c) \simeq D^c(VB_X)$, where $VB_X$ is the category of locally free sheaves
over $X$. The philosophy thus is that if one can understand the $\K$-groups for fields
and transfer maps between them, one can compute the global $\K$-groups.

However, without the regularity assumption, $\xi$ is not an equivalence and hence the
classical spectral sequences can be used to compute only the coherent $\K$-groups (better
known as $G$-groups) but not the usual $\K$-groups. This entire discussion applies for
several other generalized cohomology theories, for example triangular Witt groups \cite{BW}
and Grothendieck-Witt groups \cite{WC,S3} . When $X$ is Gorenstein, the corresponding
result for coherent Witt groups is in \cite{Gi} .

In \cite{B3}, niveau and coniveau spectral sequences are established for the usual 
$\K$-theory over a (topologically) Noetherian scheme with a bounded dimension function.
Similarly weak Gersten complexes are defined. However, the terms occuring in these
spectral sequences involve abstract derived categories with support 
$\K_i(D^b(\O_{X,x} \text{ on } \{ x \} ))$. Now, when $X$ is Cohen-Macaulay at every stalk, using Theorem \ref{intro-main}, these terms can
be identified with the $\K$-groups of an actual category of modules, that is, we can
rewrite these as $\K_i(\xbar{\P(\O_{X,x})} \cap \Mfl(\O_{X,x}))$, thus obtaining refined spectral sequences and Gersten complexes in Theorem \ref{spectral-kth-thm} . Thus,
the philosophy can now be changed to understanding the $\K$-groups of $\FP \cap \Mfl(R)$
when $R$ is a Cohen-Macaulay local ring, and maps between them.

Furthermore,the equivalence in Theorem \ref{1.1} preserves each categories' natural dualities, which induces isomorphisms of triangular Witt and Grothendieck-Witt groups, see 
Theorem \ref{witt} and Remark \ref{gw} .  Similarly, in Theorem \ref{witt-comp} we compare a recent definition \cite{MS} of triangular Witt groups for certain subcategories of triangulated categories with duality (for example $W^i(D^b_{\sFP}(\P(R))$))  . Indeed, this
article grew out of discussions between the authors on such comparisons after a series
of talks by Satya Mandal at KU.

Theorem \ref{1.1} is actually the corollary of a much more general statement involving homological dimension.  Homological dimension of a module $M$ with respect to a category $\A$, or the $\A$-dimension of $M$, is the smallest resolution of $M$ by modules in $\A$.  A resolving subcategory is in some sense the "correct" type of category over which to take homological dimension.   The following is the main theorem of this article which features as Theorem \ref{main} :
\begin{thm}\label{intro-main}
Let $\A$ be a resolving subcategory of $\M(R)$, $\bA$ the category of modules
of finite $\A$-dimension, and $\L$ a Serre subcategory satisfying condition (*), (see Definition \ref{condition*}). 
Then there is an equivalence
of derived categories
$$D^b(\bA \cap \L) \simeq D^b_{\L}(\A)$$
where $D^b_{\L}(\A)$ is the full subcategory of $D^b(\A)$ of complexes whose homologies lie in $\L$.
\end{thm}
As with the special case,  we can conclude in Theorem \ref{k-th} that the non-connective
$\K$-theory spectra are homotopy equivalent.  Also, working in this generality allows us to prove Corollary \ref{hop-nee} which generalizes an important consequence of the Hopkins-Neeman theorem \cite{H,N2}.

A brief word on the organization of the article. In Section \ref{sec2}, we give
definitions and preliminaries needed in the article. The proof of Theorem \ref{intro-main}
uses Theorem \ref{kos-lemma2} to reduce the lengths of complexes via a suitable Koszul complex and
Theorem \ref{kos-thm} which makes the reduction functorial and amenable to use in the derived category. Section \ref{sec3} is devoted to proving Theorem \ref{kos-lemma2} and Theorem \ref{kos-thm}. These theorems are crucially used in Section \ref{sec4}, where we prove the main theorems \ref{thick-main} and \ref{main}. In Section
\ref{sec5}, we use the main theorems to obtain the natural consequences for $\K$-theory, derived Witt and Grothendieck-Witt groups. In Section \ref{sec6}, we list several questions and examples of interest.

\subsection*{Acknowledgments}
We are thankful to Satya Mandal who originally suggested looking at finite projective dimension in the context of Witt groups from where this article grew. The first named author thanks him for teaching him a course on quadratic forms and the second named author for discussions related to obtaining an independent proof of the Koszul trick.
We are also thankful to Hailong Dao for many
helpful comments on several drafts and interesting conversations particularly regarding
the nature of examples in the final section \ref{sec6}. The
second named author thanks V. Srininvas and Marco Schlichting for helpful conversations, and S.M.Bhatwadekar and Anand Sawant for suggestions and comments on an initial draft.


\section{Preliminaries}
 \label{sec2}
We fix a commutative Noetherian ring $R$ throughout the article. Let $\M(R)$ denote the category of finitely generated $R$-modules. Let $\P(R)$ denote the full subcategory of projective
modules. 

 \begin{defin}
A commutative Noetherian local ring is said to be Cohen-Macaulay if its depth equals
its dimension. A commutative Noetherian ring $R$ is Cohen-Macaulay if each localization
$R_{\p}$ is Cohen-Macaulay local.
 \end{defin}

When $R$ is Cohen-Macaulay and local, we denote the full subcategory of maximal Cohen-Macaulay modules
(i.e. modules $M$ which satisfy $\text{depth}(M) = \dm(M)$) by $MCM(R)$. 

\section*{Preliminaries : Resolving, thick, and Serre subcategories}\label{cat-prelims}

Resolving subcategories allow one to generalize the notion of projective dimension and
Gorenstein dimension. For a subset $\C \subset \M(R)$ and a module $M \in \M(R)$, we say
that $\dm_\C M=n$ if $n\in\N$ is the smallest number such that there is an exact sequence
$$0 \to C_n \to \cdots\to C_0 \to M \to 0$$ with $C_0,\dots,C_n\in\C$.
For an arbitrary category $\C$, $\dm_{\C}$ does not have nice properties. However,
when $\C$ is resolving, $\dm_{\C}$ behaves similarly to projective and
Gorenstein dimension. Resolving subcategories also allow us to simultaneously study usual and coherent theory for invariants such as Grothendieck-Witt groups and $\K$-theory. They were first studied in \cite{AB} and have been of considerable interest recently. We summarize below the definitions, examples and results needed in this article.
\begin{defin}
Given a ring $R$, a full subcategory $\C \subset \M(R)$ is resolving if
\begin{enumerate}
\item $R$ is in $\C$.
\item $M \oplus N$ is in $\C$ if any only if $M$ and $N$ are in $\C$.
\item If $0 \to M'' \to M \to M' \to 0$ is exact and $M' \in \C$, then $M \in \C$
if and only if $M'' \in \C$.
\end{enumerate}
\end{defin}
Thus, $\C$ is closed under extensions and syzygies, and contains $\P(R)$.
\begin{imex}\label{res-ex}
The following categories are resolving :
\begin{enumerate}
\item $\P(R)$
\item $\M(R)$
\item Gorenstein dimension $0$ modules (also known as totally reflexive modules).
\item $\{N \mid \ext^{>n}(N,M) = 0 \forall M \in \X\}$ where 
$\X \subset \M(R)$ and $n \geq 0$
\item $\{N \mid \tor^{>n}(N,M)=0  \forall M \in \X \}$ where
$\X \subset \M(R)$ and $n \geq 0$.
\item $MCM(R)$ when $R$ is Cohen-Macaulay and local
\end{enumerate}
\end{imex}

A special class of resolving subcategories are thick subcategories. 
\begin{defin}\label{thick-def}
Let $\X \subset \M(R)$. A full subcategory $\C \subset \X$ is a thick subcategory
of $\X$ (or $\C$ is thick in $\X$) if it is resolving and for any exact sequence
with $M,M',M'' \in \X$, $$0 \to M'' \to M \to M' \to 0$$
$M'',M \in \C$ implies $M' \in \C$ too.
\end{defin}
Some authors do not require that thick subcategories contain $R$.  However, our definition is standard amongst certain authors, for example \cite{Ta}. The intuition behind requiring $R$ to be in a thick subcategory traces back to the stable category of maximal Cohen-Macaulay modules over a Gorenstein ring. More on the
nomenclature and concepts can be found in \cite{Ta}. 
\begin{lemma}
For a subset $\C \sbe \M(R)$, we let $\xbar{\C}$ denote the category of
modules $M$ such that $\dm_\C M$ is finite.  If $\C$ is resolving, then 
$$\xbar{\C}=\{M\in\M(R) \mid \Omega^{\gg 0} M \in \C\}$$ and $\xbar{\C}$ is thick in $\M(R)$
(where $\Omega^i M$ is the $i$-th syzygy of $M$).
\end{lemma}
The proof is standard and is left to the reader but can be found in \cite[Corollary 2.6 and Corollary 2.9]{Sa}.
\begin{imex}
The following categories are obtained as the thick closures in $\M(R)$ of the resolving
categories in Example \ref{res-ex} .
\begin{enumerate}
\item The category of modules with  finite projective dimension, which we will denote by 
$\xbar{\P(R)}$.
\item $\M(R)$
\item The category of modules with finite Gorenstein dimension.
\item For any $\X \subset \M(R)$, $\{Y\mid\ext^{\gg 0}(Y,X)=0\  \forall X\in\X\}$
\item For any $\X \subset \M(R)$, $\{Y\mid\tor_{\gg 0}(Y,X)=0\  \forall X\in\X\}$
\item When $R$ is Cohen-Macaulay and local, $\xbar{MCM(R)} = \M(R)$.
\end{enumerate}
\end{imex}


We define Serre subcategories so that it is convenient to talk about derived categories with supports.
\begin{defin}\label{serre-def}
A Serre subcategory $\L$ of an Abelian category is a full subcategory such that
if $0 \to M'' \to M \to M' \to 0$ is a short exact sequence in the ambient category,
then $M \in \L$ iff $M', M'' \in \L$.
\end{defin}

\begin{rem}\label{GabrielTheorem}
In fact every Serre subcategory of $\M(R)$ is obtained as the subcategory of modules
supported on a specialization-closed subset $V$ in $\spec(R)$ (i.e. if $\p \subseteq \q$
and $\p \in V$, then $\q \in V$) \cite{Ga} . For example, $\Mfl$ is obtained with $V = \mspec(R)$.  Because of this, if a module $M$ is in a Serre subcategory $\L$, then $\frac{R}{Ann(M)}$ is also in $\L$.  

Let $V$ be a closed subset of $\spec(R)$ and $c$ be any integer. The main
example we will consider in this article is the full subcategory of $\M(R)$ with modules :
$$\L = \{ M | \codm(\supp(M)) \geq c, \supp(M) \subseteq V \}.$$
\end{rem}

\begin{defin}\label{condition*}
A Serre subcategory satisfies condition (*) if  given an ideal $I$
in $R$ such that $\displaystyle{\frac{R}{I} \in \L}$, there exists a regular sequence
$\displaystyle{(a_1, a_2, \ldots, a_c) \in I}$ such that 
$\displaystyle{\frac{R}{(a_1, a_2, \ldots, a_c)} \in \L}$.
\end{defin}

\begin{imex}\label{serre-ex}
The following Serre subcategories $\L$ of $\M(R)$ satisfy condition (*) :
\begin{enumerate}
\item If $R$ is equicodimensional (i$.$e$.$ every maximal ideal has the same height) and Cohen-Macaulay, $\displaystyle{\L = \Mfl }$.
\item If $R$ is Cohen-Macaulay,
$V$ is a set theoretic complete intersection, i.e. there exists a complete
intersection ideal $J = (b_1, b_2, \ldots, b_r)$ such that $V = V(J)$, $c$ is any
integer, $\L = \{ M | \codm(\supp(M)) \geq c, \supp(M) \subseteq V \}$.
\item An important special case of the previous example is when $V = \spec(R) = V(\emptyset)$.
Then $\L$ is the category of modules supported in codimension at least $c$.
\item $V$ is a set theoretic complete intersection, $\L = \{ M | \supp(M) \subseteq V \}$.
\end{enumerate}
\end{imex}

\begin{rem}\label{2-3-rem}
We emphasize an immediate consequence of the definitions \ref{thick-def} and \ref{serre-def}.
For a thick subcategory $\T$ of $\M(R)$ and a Serre subcategory $\L$ of $\M(R)$ the
intersection $\T \cap \L$ has all the properties of a thick subcategory except possibly
that it contains $R$. In particular it has the $2$-out-of-$3$ property, that is, in a short exact sequence in $\M(R)$,
if two modules are in $\T \cap \L$, then so is the third.
\end{rem}

\textbf{
For the rest of the article, we fix a resolving subcategory $\A \subseteq \M(R)$, a
thick subcategory $\T$ of $\M(R)$ and a Serre subcategory $\L$ of $\M(R)$ satisfying
condition (*).}

\begin{defin}
Let $\mathcal{E}\sbe \M(R)$ be a subcategory closed under extensions (e$.$g$.$ $\A$, $\T$, $\T\cap\L$).  Let $\ch^b_{\L}(\mathcal{E})$ be the subcategory of $\ch^b(\mathcal{E})$  consisting of complexes whose homologies lie in $\L$, and let  $D^b_{\L}(\mathcal{E})$  be the subcategory of $D^b(\mathcal{E})$ consisting of objects whose homologies lie in $\L$.
\end{defin}
Since $\L$ is a Serre subcategory, it is easy to see that $D^b_\L(\mathcal{E})$ is a triangulated subcategory of $D^b(\mathcal{E})$.  Furthermore, it is easy to show that $D^b_{\L}(\mathcal{E})$ is equivalent to the category obtained by inverting the quasi-isomorphisms of $\ch^b_{\L}(\mathcal{E})$. 

\begin{rem}\label{recall}
We briefly mention what these definitions will help us achieve :
 
There is a natural functor $\ch^b(\T \cap \L) \rightarrow \ch^b_{\L}(\T)$ which
induces the natural functor $D^b(\T \cap \L) \rightarrow D^b_{\L}(\T)$. In
Theorem \ref{thick-main}, we will prove that this is an equivalence.

When $\T = \bA$ where $\A$ is resolving, there is a natural
functor $\ch^b_{\L}(\A) \rightarrow \ch^b_{\L}(\T)$ which induces an equivalence
$D^b_{\L}(\A) \stackrel{\sim}{\rightarrow} D^b_{\L}(\T)$.
\end{rem}

To a reader who maybe lost in the notation, we highlight the two important special
cases which might shed further light on what this achieves. Let $R$ be equicodimensional and Cohen-Macaulay.
\begin{enumerate}
\item $\A = \P(R)$, $\T = \bA = \text{finite projective dimension modules},
\L = \Mfl$. Then $\FP \cap \Mfl = \T \cap \L$ and hence
the equivalences in \ref{recall} yield $D^b_{fl}(\P(R)) \simeq D^b(\FP \cap \Mfl)$.
\item $\A = \M(R)$, $\T = \bA = \M(R), \L = \Mfl$. Then
the comparison in \ref{recall} yields the well-known equivalence $D^b_{fl}(\M(R)) \simeq 
D^b(\Mfl)$.
\end{enumerate}


\section*{Preliminaries : Chain complexes in $\A$}\label{chain-prelims}
In this section, we state and prove some general lemmas for chain complexes of
resolving (and thick) subcategories and their derived categories. We begin by introducing
some notations for chain complexes.
\begin{defin}
Let $P_{\bullet} \in \ch^b(\M(R))$, i.e. chain complexes with elements in $\M(R)$.
\begin{enumerate}
\item $\min_c(P_{\bullet})$ is defined as $\sup \{ n \mid P_i = 0 \quad \forall~ i < n\}$.
\item $\max_c(P_{\bullet})$ is defined as $\inf \{ n \mid P_i = 0 \quad \forall~ i > n\}$.
\item $\min(P_{\bullet}) = \sup\{ n \mid H_i(P_{\bullet})=0 \quad \forall i<n\}$
\item $\max(P_{\bullet}) = \inf\{ n\mid H_i(P_{\bullet})=0 \quad \forall i>n\}$
\item $\wid(P_{\bullet}) = \max (P_{\bullet}) -\min(P_{\bullet})$ if $P_{\bullet}$ is not acyclic and 
$\wid(P_{\bullet})=0$ if it is acyclic.
\item $Z^P_n = \ker(\partial_n), B^P_n = \partial_{n+1}(P_{n+1})$.
\item $\supph(P_{\bullet}) = \{ n \mid H_n(P_{\bullet}) \neq 0 \}$.
\item $T$ denotes the shift functor.
\end{enumerate}
\end{defin}

The following lemma allows us to assume that the complexes we will work with have certain properties, simplifying later proofs.
\begin{lemma}\label{change-proj}
\begin{enumerate}

\item\label{statement1} Let $P_{\bullet} \in \ch^b(\T \cap \L)$ and $\min(P_{\bullet}) \geq m$. Then there exists a quasi-isomorphism
$P'_{\bullet} \xrightarrow{\phi} P_{\bullet}$ with $P'_{\bullet} \in \ch^b(\T \cap \L)$,
$\min_c(P'_{\bullet}) =m$.

\item\label{statement2} Let $P_{\bullet} \in \ch^b_{\L}(\T)$. Let $t$ be any natural number and $m$ be an
integer such that $\min(P_{\bullet}) \geq m$. Then there exists a quasi-isomorphism
$U_{\bullet} \xrightarrow{\phi} P_{\bullet}$ with $U_{\bullet} \in \ch^b_{\L}(\T)$
such that $U_i$ is free for all   $i < t$ and $\min_c(U_{\bullet}) = m$.

\item\label{statement3} Let $P_{\bullet} \in \ch^b_{\L}(\T)$. Let $m$ be an
integer such that $\min(P_{\bullet}) \geq m$. Then there exists a quasi-isomorphism
$U_{\bullet} \xrightarrow{\phi} P_{\bullet}$ such that $U_i$ is free for all   $i$.

\item\label{statement4} Let $\mathcal{E}$ be either $\T$ or $\T\cap \L$.  For $P_\bullet\in Ch^b(\mathcal{E})$, if $\supph(P_\bullet)=\{m\}$, then $P_\bullet$ is isomorphic in $D^b(\mathcal{E})$ to the complex $T^mH_m(P_\bullet)$ in $D^b(\mathcal{E})$.  
\end{enumerate}
\end{lemma}
\begin{proof}
First, we prove (1). Consider the complex 
\[P'_{\bullet} : \cdots \rightarrow P_{m+2} \rightarrow P_{m+1} \rightarrow
Z^P_m \rightarrow 0 \cdots.\] Since $\min(P_{\bullet}) \geq m$, the complex $$0\to Z^P_m\to P_m\to \cdots \to P_{\min_C(P_{\bullet})}\to 0$$
is acyclic.  Since each $P_i\in \T \cap \L$ and  $\T \cap \L$ has the $2$-out-of-$3$ property by Remark \ref{2-3-rem}, we conclude that $Z^P_m$ is also in $\T \cap \L$.  Hence $P'_{\bullet}$ is in $\ch^b(\T \cap \L)$.   Furthermore, 
the inclusion map $P'_{\bullet} \xrightarrow{} P_{\bullet}$ is a quasi-isomorphism.  By construction, $\min_c(P'_{\bullet}) =m$ which completes the proof.  

Now we prove (2). The statement is clearly true when $P_{\bullet}$ is acyclic. Assume now that $P_{\bullet}$ is not acyclic. The same arguments as in the previous paragraph show that there exists a quasi-isomorphism $P'_{\bullet} \xrightarrow{\phi'} P_{\bullet}$ with $P'_{\bullet} \in \ch^b_\L(\T)$,
$\min_c(P'_{\bullet}) =m$.  Set $n =t-m$. For each $i$, choose free modules $Q_{i,j}$ for $0\le j\le n$   such that
$$0\to \Omega^n P'_i\to Q_{i,n}\to \cdots \to Q_{i,0}\to P'_i\to 0$$
is exact.  Setting $Q_{i, n+1} = \Omega^n P_i$ and $Q_{i,l}=0$ for all $j \ne 0,\dots,n+1$ gives us a complex $Q_{i\bullet}$ which is a resolution of $P'_i$. The differentials of $P'_{\bullet}$ lift to give a double
complex $Q_{\bullet \bullet}$. Let $U_{\bullet}$ be the total complex  of $Q_{\bullet \bullet}$. There is a quasi-isomorphism $U_\bullet\xrightarrow{\alpha}P'_{\bullet}$. Since $\T$ is thick, each $\Omega^n P'_i$ is in $\T$, and so $U_\bullet \in \ch^b_{\L}(\T)$.  The composition $U_\bullet\xrightarrow{\alpha} P'_\bullet\xrightarrow{\phi'}
P_\bullet$ is the desired quasi-isomorphism.

Statement (3) is standard and follows from the proof of (2) by taking $t=\infty$.  

Lastly, we show (4). By the proof of statements (1) and (2), there exists a quasi isomoprhism $P'_\bullet\rightarrow P_\bullet$ such that $\min_c(P'_{\bullet})=m$.  Since both  $\T$ and $\L$ have the $2$-out-of-$3$ property (by Remark \ref{2-3-rem}), $H_m(P'_\bullet)$ is in $\mathcal{E}$. Since $Z^{P'}_m=P'_m$, we have a map $P'_m\to H_m(P'_\bullet)\cong H_m(P_\bullet)$. This extends to a quasi-isomorphism $P'_\bullet\to T^m H_m(P_\bullet)$.  This completes the proof.

\end{proof}

We now state a result on morphisms in $D^b_{\L}(\T)$.
\begin{lemma}\label{hom-set-0}
Let $P_{\bullet},Q_{\bullet} \in D^b_{\L}(\T)$ such that $\min(P_{\bullet}) > \max(Q_{\bullet})$.
Then we have $\hm_{D^b_{\L}(\T)}(P_{\bullet},Q_{\bullet}) = 0$.
\end{lemma}
\begin{proof}
The result is clear when $P_{\bullet}$ is acyclic, and hence we assume that $P_{\bullet}$ is not acyclic.
Let $f \in \hm_{D^b_{\L}(\T)}(P_{\bullet},Q_{\bullet})$ be represented
by a roof diagram 
$$P_{\bullet} \xleftarrow{q} P'_{\bullet} \xrightarrow{g} Q_{\bullet}$$
where $q$ is a quasi-isomorphism. Set $t=\max_c(Q_{\bullet})$ and $m=\min(P_{\bullet})$.   By Lemma \ref{change-proj}(\ref{statement2}), we can assume that
$\min_c(P'_{\bullet}) = m$ and $P'_i$ is free for all $t\ge i\ge m$. Set $h_i=0:P'_{i}\to Q_{i+1}$ for all $i$ not in $[t,m]$. Let $\partial^Q$ be the differential of $Q_\bullet$.
Since for all $t\ge i\ge m$, $P'_i$ is free and 
$m>\max(Q_\bullet)$, $$Q_{i+1}\xrightarrow{\partial^Q_{i+1}} \ker( \partial^Q_{i})$$ is surjective, we can define maps $h_i:P'_i\to Q_{i+1}$  such that $h=\{h_i\}$ is a null-homotopy for $g$. Hence  $g$ is null-homotopic and $f = 0$.
\end{proof}

We end this section with a general result on triangulated categories.

\begin{lemma}\label{FinalEquivalence}
Let $\mathcal{U},\mathcal{V}$ be triangulated categories and $\xi:\mathcal{U}\to\mathcal{V}$ a triangulated functor.  If $\xi$ is full, essentially surjective, and faithful on objects, then $\xi$ is an equivalence of categories.
\end{lemma}
\begin{proof}
It suffices to show that $\xi$ is faithful. Suppose $f : X \to Y$ is a morphism in $\mathcal{U}$ such that $\xi(f) = 0$.
Letting $T$ be the shift functor, complete $f$ to a triangle 
$$C \xrightarrow{g} X \xrightarrow{f} Y \xrightarrow{h} TC.$$
In $\mathcal{V}$, this maps to the triangle
$$\xi(C) \xrightarrow{\xi(g)} \xi(X) \xrightarrow{0} \xi(Y) \xrightarrow{\xi(h)} T\xi(C).$$
So, by \cite[Corollary 1.2.7]{Neeman}, there exists an $\eta  :\xi(X) \xrightarrow{} \xi(C)$
such that $\xi(g) \circ \eta$ is the identity on $\xi(X)$.  
Since $\xi$ is full, there exists a morphism  $\tilde{\eta}: X\to C$ such that  $\xi(\tilde{\eta})=\eta$.
Extending the morphism $g \circ \tilde{\eta}$ to an exact triangle in $\mathcal{U}$ gives us the triangle
$$ X \xrightarrow{g\circ \tilde{\eta}} X \to W \to TX.$$
Since $\xi(g\circ \tilde{\eta})$ is the identity on $\xi(X)$, the image of the above triangle in $\mathcal{V}$ is
$$ \xi(X) \xrightarrow{\text{Id}_{\xi(X)}}\xi(X) \to \xi(W) \to T\xi(X).$$
This means that $\xi(W)\cong 0$ in $\mathcal{V}$. But since $\xi$ is faithful on objects,
$W \cong 0$ in $\mathcal{U}$.  Hence $g \circ \tilde{\eta}$ is an isomorphism
in $\mathcal{U}$.  So $\tilde{\eta} \circ (g \circ \tilde{\eta})^{-1} : X \to C$
gives a splitting, i$.$e$.$ $g\circ\tilde{\eta} \circ (g \circ \tilde{\eta})^{-1}$ is the identity on $X$. Therefore, by \cite[Corollary 1.2.8]{Neeman}, the triangles 
$$C \xrightarrow{g} X \xrightarrow{f} Y \xrightarrow{h} TC$$
$$T^{-1} Y\oplus X\rightarrow X \xrightarrow{0} Y\rightarrow Y\oplus TX$$
are isomorphic in $\mathcal{U}$.  Hence $f = 0$ which completes the proof.
\end{proof}


\section*{Preliminaries : Semidualizing modules}\label{semidual}
Semidualizing modules seem to have arisen in work of Foxby and we refer to the
exposition \cite{S} by Sather-Wagstaff for more details. We will require them only when we
consider dualities in the context of Witt groups in Section \ref{sec5} and the reader
can skip this subsection presently and refer back later.

\begin{defin}
A module $C$ is semidualizing if $\ext^{>0}(C,C)=0$ and the natural map
$R \iso \hm(C,C)$ is an isomorphism.
\end{defin}
We fix a semidualizing module $C \in \M(R)$ and write $M^{\da}=\hm(M,C)$ for the rest of
this subsection.
\begin{defin}
A finitely generated module $M$ is totally $C$-reflexive if it satisfies the following
\begin{enumerate}
\item $\ext^{>0}(M,C)=0$
\item  $\ext^{>0}(M^\da,C)=0$
\item The natural evaluation map $\eta_M:M\to M^{\da\da}$ is an isomorphism.
\end{enumerate}
Let $\G_C$ denote the category of totally $C$-reflexive modules. Then $\G_C$ is a
resolving subcategory.
\end{defin}
\begin{imex}
Some examples of semidualizing modules $C$ and the corresponding categories $\G_C$ are :
\begin{enumerate}
\item When $C=R$, $\G_R$ is simply the category of Gorenstein dimension $0$ modules
(also known as totally reflexive modules).
\item When $R$ is local and $C = D$, that is, the dualizing (or canonical) module,
then $\G_D = MCM(R)$.
\end{enumerate}
\end{imex}

\begin{lemma}\label{cdim}
If $M \in \xbar{\G_C}$, then $\dm_{\G_C} M=\min\{n\mid \ext^{>n}(M,C)=0\}$.
\end{lemma}


\section{The Koszul Construction}
 \label{sec3}

We begin this section with a theorem based on a very interesting construction on Koszul complexes which essentially first appears in an unpublished preprint \cite{F}. The construction is used in \cite[Lemma 1]{RS} and is generalized in \cite{FH}.  \textbf{Recall that $\T \subseteq \M(R)$ is a thick subcategory, and
$\L \subseteq \M(R)$ is a Serre subcategory satisfying condition (*).}

\begin{thm}\label{kos-lemma2}
Let $P_{\bullet} \in \ch^b_{\L}(\M(R))$ and with $\min(P_{\bullet}) \geq m$. Fix an ideal $J\subseteq R$ such that $\displaystyle{\frac{R}{J}\in \L}$.  There exists a regular sequence $g_1,\ldots,g_c\in J$ such that, after setting 
\[K_{\bullet} = \kos(f_1, f_2, \ldots, f_c) \otimes_R F,\]
for some free module $F$, there exists a chain map $\alpha:K_\bullet\to P_\bullet$
such that the following hold:
\begin{itemize}
\item $K_\bullet \in \ch^b_{\L}(\P(R))$
\item $\min_c(K_\bullet)=m$
\item $\supph(K_{\bullet}) = \{ m \}$
\item $H_m(\alpha):H_m(K_\bullet)\twoheadrightarrow H_m(P_\bullet)$ is surjective.

\end{itemize}
\end{thm}
The proof of our main result involves several technical subtleties, thus we require $\alpha$ and $K_\bullet$ to have these specific properties. Since neither \cite{F,RS,FH} give the exact statement needed, we give the following proof.  
\begin{proof}
By Remark \ref{GabrielTheorem}, we know that there exists a specialization closed subset $V \subseteq \spec(R)$ such that
\[\L=\{M\in\M(R)\mid \supp(M)\subseteq V\}.\]
Without loss of generality, we may assume that $m=0$.  

By Lemma \ref{change-proj}(\ref{statement3}), there exists a quasi-isomorphism   $\pi:F_\bullet\to P_\bullet$ with each $F_i$ free and $F_i=0$ for all $i<0$.  
By repeated application of the dual of \cite[Corollary 10.4.7]{Wei}, we get the following natural isomorphisms of $\hm$ spaces :
$$\End_{K(R)}(F_\bullet) \cong 
\End_{D(R)}(F_\bullet) \cong
\hm_{D(R)}(F_\bullet, P_\bullet) \cong
\hm_{K(R)}(F_\bullet, P_\bullet).
$$
Since $P_\bullet$ is a bounded complex, 
$\hm_{K(R)}(F_\bullet, P_\bullet)$ is a subquotient of a finitely generated module and hence is finitely generated. Hence, $\End_{K(R)}(F_\bullet)$ is a finitely generated $R$-module. Similarly using the above isomorphisms and the boundedness of $P_\bullet$, we see that the map $(\End_{K(R)}(F_\bullet))_p \to  \End_{K(R_p)}((F_\bullet)_p)$ is injective for every $p\in \spec(R)$.

Assume that $(\End_{K(R)}(F_\bullet))_p\ne 0$ for some $p\in \spec(R)$. Then $\End_{K(R_p)}((F_\bullet)_p)\ne0$ which implies that $(F_\bullet)_p$ is not homotopic to $0_\bullet$.  Since $F_\bullet$ is a bounded below complex of free modules,  $(F_\bullet)_p$ is not acyclic, and so  $H_i(F_\bullet)_p\ne 0$ for some $i\in\mathbb{N}$. Therefore we have
\[\supp\left(\End_{K(R)}(F_\bullet)\right)\subseteq\bigcup_{i=0}^\infty \supp\left(H_i(F_\bullet)\right)=\bigcup_{i=0}^\infty \supp\left(H_i(P_\bullet)\right)\subseteq V\]
where the last containment follows from the assumption that $P_\bullet\in \ch^b_{\L}(\M(R))$.  Set $I=Ann\left(\End_{K(R)}(F_\bullet)\right)$.  We have
\[\supp\left(\frac{R}{I\cap J}\right)=\supp\left(\frac{R}{I}\right)  \bigcup  \supp\left(\frac{R}{J}\right)=\supp\left(\End_{K(R)}(F_\bullet)\right) \bigcup \supp\left(\frac{R}{J}\right)\subseteq V.\]
It follows that $\displaystyle{\frac{R}{I\cap J}\in \L}$. 

Invoking condition (*), there exists a regular sequence $f_1,\dots,f_c\in I\cap J$
such that $\displaystyle{\frac{R}{(f_1, f_2, \ldots, f_c)}\in \L}$. We show that this is our desired regular sequence. Set 
\[K_{\bullet} = \kos(f_1, f_2, \ldots, f_c) \otimes_R F_0.\]
For each $j\in\{1,\dots,c\}$, since $f_j\in I$, $f_j\id_{F_\bullet}$ is null-homotopic with null-homotopy $\alpha_j$.  Letting $-^*$ denote the $R$-dual of a complex, it follows that $f_j\id_{F^*_\bullet}$ is also null-homotopic with null-homotopy ${\alpha_j}^*$.  

The result  \cite[Proposition 23]{FH} states that when $F_\bullet$ is bounded above, there exists a chain map $\phi:F^*_\bullet \rightarrow K^*_\bullet$ such that $\phi_0$ is an isomorphism (in the language of \cite{FH}, $\left({\alpha_1}^*,\dots,{\alpha_c}^*\right)$ is an $S$-contraction with weight  $(f_1,f_2,\dots,f_c)$). However, a close examination of the proof reveals that the bounded assumption is essential (the authors only use the bounded assumption earlier in the paper to construct an $S$-contraction).  Thus, even though  $F_\bullet$ is not bounded above, such a $\phi$ exists.  

Set $\tilde{\alpha}=\phi^*:K_\bullet\to F_\bullet$.  Since $\tilde{\alpha}_0$ is an isomorphism, and since $K_{-1}=F_{-1}=0$, the commutative diagram
\[\xymatrix{
K_0 \ar[r]_{\tilde{\alpha}_0}^{\sim} \ar@{->>}[d] &F_0 \ar@{->>}[d] \\
H_0(K_\bullet) \ar[r]_{H_0(\tilde{\alpha})} & H_0(F_{\bullet}) \\
}
\]
shows that $H_0(\tilde{\alpha}):H_0(K_\bullet)\to H_0(F_{\bullet})$ is surjective.  Set $\alpha=\pi\circ\tilde{\alpha}:K_\bullet\to P_\bullet$.  Since $\pi$ is a quasi-isomoprhism, it is clear that $\alpha$ and $K_\bullet$ have the desired properties.

\end{proof}

The following result will also be useful in applying
Theorem \ref{kos-lemma2} in the proof of Proposition
\ref{ess-surj-full}.

\begin{lemma}\label{triang-ineq2}

Assume the set up of Theorem \ref{kos-lemma2} and consider the chain map $\alpha:K_\bullet\to P_\bullet$ constructed there. Extend this chain map to the exact triangle
\[T^{-1} C_\bullet\rightarrow K_\bullet\xrightarrow{\alpha} P_\bullet\to C_\bullet.\]
If $\wid(P_\bullet)>0$, then we have the following:
\begin{enumerate}
\item\label{5.alpha}   $\wid(C_{\bullet})<\wid(P_\bullet)$
\item\label{5.beta}  $\wid(T^{-1}C_{\bullet}\oplus K_\bullet)<\wid(P_\bullet)$.
\end{enumerate}
\end{lemma}
\begin{proof}
Set $m=\min(P_\bullet)$.  Since $H_i(K_{\bullet})=0$ for all $i\neq m$, we have $H_i(P_{\bullet})\cong H_i(C_{\bullet})$ for all $i\ne m,m+1$. Furthermore, $H_m(\alpha)$ is surjective 
so $H_m(C_{\bullet}) = 0$. 
The above statements now follow
from these observations.
\end{proof}

In other works such as \cite{F},\cite{FH}, and \cite{RS}, the authors use results similar to Theorem \ref{kos-lemma2} to give isomorphisms of Grothendieck groups.  Since we intend to prove an equivalence of categories,  we need a version of Theorem \ref{kos-lemma2} for morphisms of complexes.  

\begin{thm}\label{kos-thm}
Let $X_{\bullet} \xrightarrow{g} Y_{\bullet}$ be a morphism in $D^b_{\L}(\T)$ 
such that $X_{\bullet}, Y_{\bullet}$ are complexes in $\ch^b(\T \cap \L)$,
$\min(X_{\bullet} \oplus Y_{\bullet}) = m$, and $\min_c (X_\bullet),\min_c(Y_\bullet)\ge m$. Then there exist complexes
$M^X_{\bullet}$ and $M^Y_{\bullet}$ in $\ch^b(\T \cap \L)$ and chain maps
$M^X_{\bullet} \xrightarrow{\beta^X} X_{\bullet}$,
$M^Y_{\bullet} \xrightarrow{\beta^Y} Y_{\bullet}$ 
and $M^X_{\bullet} \xrightarrow{g'} M^Y_{\bullet}$ such that 
\begin{itemize}
\item $M^X_{\bullet}, M^Y_{\bullet} \in \ch^b(\T \cap \L)$ with $M^X_i=M^Y_i=0$ for all $i\ne m$.
\item there is a commutative square in $D^b_{\L}(\T)$ :
\[
\xymatrix{
M^X_{\bullet} \ar[r]^{\beta^X} \ar[d]_{\kappa} & X_{\bullet} \ar[d]^g \\
M^Y_{\bullet} \ar[r]^{\beta^Y} & Y_{\bullet} \\
}
\]
\item $H_m(\beta^X)$ and $H_m(\beta^Y)$ are surjective.
\end{itemize}
\end{thm}

\begin{proof}
Let $M^Y_\bullet$ be the complex $T^m Y_m$. Then $M^Y_\bullet \in \ch^b(\T \cap \L)$ and clearly there is a chain complex morphism $\beta^Y : M^Y_\bullet \to Y_\bullet$. Since $H_m(M^Y_\bullet) = Y_m$ and $H_m(Y_\bullet)$ is a quotient of $Y_m$, we obtain that $H_m(\beta^Y)$ is surjective.

Let $g$ be given by a roof diagram 
$X_{\bullet} \xleftarrow{q} Q_{\bullet} \xrightarrow{f} Y_{\bullet}$ where $q$ is a quasi-isomorphism. By Lemma \ref{change-proj}(\ref{statement2}), we may assume $\min_c(Q_{\bullet}) = m$. Let $Q'_\bullet$ be the pull-back
of $\alpha^Y$ and $f$ in $\ch^b(\M(R))$ :
\[
\xymatrix{
Q'_{\bullet} \ar@{-->}[r]^{\nu} \ar@{-->}[d]_{\mu} & Q_{\bullet} \ar[d]^f \\
M^Y_{\bullet} \ar[r]^{\beta^Y} & Y_{\bullet}\\
}
\]
 By definition of the pull-back, there is a left exact sequence of complexes :
\[
0_\bullet \to Q'_\bullet \to Q_\bullet \oplus M^Y_\bullet \xrightarrow{(f,-\beta^Y)} Y_\bullet .
\]
Let $B_\bullet$ be the image of $(f,-\alpha^Y)$. Since $\L$ is closed under subobjects and each $B_i\sbe Y_i\in\L$, we have $B_\bullet \in \ch^b(\L)$.  Furthermore, since $\L$ is also closed under quotients, $B_\bullet\in \ch^b_{\L}(\M(R))$.  Since the homologies of $Q_\bullet \oplus M^Y_\bullet$ also lie in $\L$, the  short exact sequence 
\[
0_\bullet \to Q'_\bullet \to Q_\bullet \oplus M^Y_\bullet \to B_\bullet \to 0_\bullet
\]
shows that  
$Q'_\bullet \in \ch^b_{\L}(\M(R))$.  Set $\lambda = q \circ \nu$. 

Set $J=Ann(X_m)$.  By Remark \ref{GabrielTheorem}, $\displaystyle{\frac{R}{J}}\in\L$. By Theorem \ref{kos-lemma2}, there exists a regular sequence $x_1, x_2, \ldots, x_c\in J$ and a chain map $\alpha^X: K^X_{\bullet} \rightarrow Q'_{\bullet}$ with the following properties :
\begin{itemize}
\item $K^X_{\bullet} = \kos(x_1, x_2, \ldots, x_c) \otimes_R R^n \in \ch^b_{\L}(\P)$
\item $\min_c(K^X_{\bullet}) = m$
\item $\supph(K^X_{\bullet}) = \{ m \}$
\item $H_m(\alpha^X):H_m(K_\bullet^X)\rightarrow H_m(Q'_\bullet)$ is surjective .
\end{itemize}
Thus, we obtain the following commutative diagram in $D^b_{\L}(\M(R))$
\[
\xymatrix{
K^X_{\bullet} \ar[r]^{\alpha^X} & Q'_{\bullet} \ar[d]_{\mu} \ar[r]^{\lambda} &  X_{\bullet} \ar[d]^{g} \\
& M^Y_{\bullet} \ar[r]^{\beta^Y} &  Y_{\bullet} \\
}
\]

Let $M^X_{\bullet}$ be the chain complex $T^mH_m(K^X_{\bullet})
\in \ch^b(\bar{\P} \cap \L) \subseteq \ch^b(\T \cap \L) \subseteq \ch^b_{\L}(\T)$. Since  $$\frac{K^X_m}{(x_1,\dots,x_c)K^X_m}\cong H_m(K^X_\bullet)$$ and $x_1,\dots,x_c\in J=Ann(X_m)$, the map $\lambda_m\circ\alpha_m^X:K_m^X\to X_m$ factors through $H_m(K_\bullet^X)$ giving us the homomorphism $\beta^X_m:H_m(K^X_\bullet)\to X_m$. Also, since $Z^{K^X}_m=K_m^X$, the image of $\lambda_m\circ\alpha_m^X$, and hence the image of $\beta^X_m$, lies in $Z_m^X$. Therefore, we have actually defined a chain map $\beta^X:M^X_\bullet\to X_\bullet$. Similarly, it is also
clear that $\mu \circ \alpha^X$ factors through $M^X_\bullet$ giving us the chain map $\kappa:M^X_\bullet \rightarrow M^Y_\bullet$.   Hence we have the required  commutative diagram in $D^b_{\L}(\T)$ :
\[
\xymatrix{
M^X_{\bullet} \ar[r]^{\beta^X} \ar[d]_{\kappa} & X_{\bullet} \ar[d]^g \\
M^Y_{\bullet} \ar[r]^{\beta^Y} & Y_{\bullet} \\
}
\]

It remains to show that $H_m(\beta^X)$ is surjective. Since $q$ is a quasi-isomorphism and $H_m(\alpha^X)$ is surjective, in order to show that 
$H_m(\beta^X)=H_m(q)\circ H_m(\nu)\circ H_m(\alpha^X)$
is surjective, it suffices to show that $H_m(\nu)$ is surjective.

Since $\beta^Y_m$ is an isomorphism and $\beta^Y_{m+1}=0$, $Q'_m \xrightarrow{\nu_m} Q_m$ is an isomorphism and $Q'_{m+1} \xrightarrow{\nu_{m+1}} Q_{m+1}$ is injective. Since $Q'_i=Q_i=0$ for $i < m$, we have the following morphism of exact sequences:
\[
\xymatrix{
Q'_{m+1} \ar[r] \ar[d]_{\nu_{m+1}}
& Q'_{m} \ar[r] \ar[d]_{\nu_m}^{\wr} 
& H_m(Q'_\bullet) \ar[r] \ar[d]_{H_m(\nu)}
& 0 \\
Q_{m+1} \ar[r]
& Q_{m} \ar[r]
& H_m(Q_\bullet) \ar[r]
& 0\\
}.\]
Thus, $H_m(\nu)$ is surjective. This completes the proof.

\end{proof}

The following result will also be useful in applying Theorem \ref{kos-lemma2} in the
proof of Proposition \ref{ess-surj-full}.

\begin{lemma}\label{triang-ineq}
Assume the hypothesis of Theorem \ref{kos-thm} and consider the commutative diagram constructed
\[
\xymatrix{
M^X_{\bullet} \ar[r]^{\beta^X} \ar[d] & X_{\bullet} \ar[d]^g \\
M^Y_{\bullet} \ar[r]^{\beta^Y} & Y_{\bullet}\\
}
\]
Suppose
$\wid(X_{\bullet} \oplus Y_{\bullet}) = k > 0$. Complete $\beta^X$ and $\beta^Y$ to
exact triangles :
\[
T^{-1}C^X_{\bullet} \to M^X_{\bullet} \xrightarrow{\beta^X} X_{\bullet} \xrightarrow{\gamma^X} C^X_{\bullet}   \qquad
T^{-1}C^Y_{\bullet} \to M^Y_{\bullet} \xrightarrow{\beta^Y} Y_{\bullet} \xrightarrow{\gamma^Y} C^Y_{\bullet}
\]
Then we have the following :
\begin{enumerate}
\item\label{5.gamma}  $\wid (C^X_{\bullet} \oplus C^Y_{\bullet}) < k$
\item\label{5.epsilon} $\wid(C^X_\bullet\oplus Y_\bullet)\le k$
\item If $\min(X_\bullet)<\min(Y_\bullet)$, then $\min(C^X_\bullet\oplus Y_\bullet)\le k-1$.
\end{enumerate}
\end{lemma}
\begin{proof}
The proof is straightforward and the same arguments in the proof of Lemma \ref{triang-ineq2} imply the above statements.
\end{proof}

\section{The Equivalence of Derived Categories}
 \label{sec4}
In this section, we prove the promised equivalence of the derived
categories $D^b(\T \cap \L)$ and $D^b_{\L}(\T)$ and then consider the case where
$\T = \bA$, yielding our main theorem. We begin by defining the natural functor.

\begin{defin}
Let $\iota : D^b(\T \cap \L) \rightarrow D^b_{\L}(\T)$ be the natural functor
induced by the inclusion $\ch^b(\T\cap\L)\hookrightarrow \ch^b_\L(\T)$. 
Abusing notation, we will sometimes write $X_{\bullet}\in D^b_{\L}(\T)$ when we mean $\iota(X_{\bullet})$. Similarly,
for a morphism $f \in D^b_{\L}(\T)$, we will write $f \in D^b(\T \cap \L)$ when we mean $f = \iota(g)$ for some $g \in D^b(\T \cap \L)$.


\end{defin}

\begin{lemma}\label{hom-set-same}
Viewing the modules $M, N \in \T \cap \L$ as complexes, we have
{\scalefont{0.85}
\[
\hm_{\T \cap \L}(M,N) \iso \hm_{D^b(\T \cap \L)}(M,N) \iso \hm_{D^b_{\L}(\T)}(\iota(M),\iota(N)) \iso \hm_{R}(M,N).
\]
}
\end{lemma}
\begin{proof}
The maps are induced by the functors
\[\T\cap\L\to D^b(\T\cap\L)\xrightarrow{\iota} D^b_{\L}(\T)\xrightarrow{H_0} \M(R)\]
where the last functor is  the zeroth homology functor.
The composition of the maps is identity on $\hm_{R}(M,N)$. Thus the first map is injective. Let
$\beta \in \hm_{D^b(\T \cap \L)}(M,N)$. Let $\beta$ be given by the roof diagram
$M \xleftarrow{q} T_{\bullet} \xrightarrow{f} N$ where $q$ is a quasi-isomorphism.
By Lemma \ref{change-proj}(\ref{statement1}), there is a quasi-isormophism $T'_{\bullet} \xrightarrow{i} T_{\bullet}$
such that $\min_c(T'_{\bullet}) = 0$.
Then $H_0(T_{\bullet}) = H_0(T'_{\bullet})$ and let 
$T'_{\bullet} \xrightarrow{\mu} H_0(T_{\bullet})$ be the obvious map. Then
the following diagram commutes :
\[
\xymatrix{
& & & T_{\bullet} \ar[llld]^{q}_{\sim} \ar[rrrd]^{f} & & &\\
 M & & & T'_{\bullet} \ar[lll]^{q \circ i}_{\sim} \ar[rrr]^{f \circ i}
 \ar[u]^{i}_{\wr} \ar[d]_{H_0(q) \circ \mu}^{\wr} & & & N \\
& & & M \ar@{=}[ulll] \ar[urrr]_{H_0(f) \circ {H_0(q)}^{-1}} & & &\\
}
\]
This tells us that $\beta$ is equivalent to the map induced by 
$H_0(f) \circ {H_0(q)}^{-1}$. So the first map is an isomorphism which forces the second
map to be an injection. The same argument shows that the map \[\hm_{\T\cap\L}(M,N)\rightarrow \hm_{D^b_{\L}(\T)}(\iota(M),\iota(N))\] is surjective, proving the claim.
\end{proof}

\begin{lemma}\label{sp-inj}
Let $X_{\bullet}, Y_{\bullet} \in D^b(\T \cap \L)$ such that $\supph(Y_{\bullet}) = \{ m \}$
and $\min(X_{\bullet}) \geq m$. Then 
$\hm_{D^b(\T \cap \L)}(X_{\bullet}, Y_{\bullet}) \rightarrow \hm_{D^b_{\L}(\T)}
(X_{\bullet}, Y_{\bullet})$ is injective.
\end{lemma}
\begin{proof}
We may assume that $m=0$.  By Lemma \ref{change-proj}(\ref{statement4}), 
 $Y_{\bullet}$ is isomorphic in $D^b(\T\cap\L)$ to the complex $U_\bullet = H_0(Y_\bullet)\in\T\cap\L$.  Thus we choose to work with $U_\bullet$. Let $g \in \hm_{D^b(\T \cap \L)}(X_{\bullet}, U_\bullet)$. Then $g$ is given by a roof diagram 
$X_{\bullet} \xleftarrow{q} W_{\bullet} \xrightarrow{f} U_\bullet$ where $q$ is a
quasi-isomorphism. It is enough to show that $\iota(f) = 0$ implies $f = 0$.
By Lemma \ref{change-proj}(\ref{statement1}), we may assume that $\min_c(W_{\bullet}) = \min(X_{\bullet}) \geq 0$.
Note that $\iota(f) = 0$ implies $H_0(f) = 0$. Since $\min_c(W_{\bullet}) \geq 0$, there is
a natural surjection $W_0 \stackrel{h}{\twoheadrightarrow} H_0(W_{\bullet})$. Further, since 
$U_i = 0$ for all $i \ne 0$, $f_0=  H_0(f) \circ h = 0$. This finishes the proof.
\end{proof}

We now have all the ingredients to prove the equivalence of the derived categories
$D^b(\T \cap \L)$ and $D^b_{\L}(\T)$.
\begin{propn}\label{ess-surj-full}
The functor $\iota$ is essentially surjective and full.
\end{propn}
\begin{proof}
We prove the following set of statements by induction on $k$ : 
\begin{enumerate}
\item For every $P_{\bullet} \in D^b_{\L}(\T)$, with 
$\wid (P_{\bullet}) = k$, there exists $\tilde{P}_{\bullet} \in D^b(\T \cap \L)$ such that $\iota(\tilde{P}_{\bullet}) \cong P_{\bullet}$.
\item For every $X_{\bullet}, Y_{\bullet} \in D^b(\T \cap \L)$ with 
$\wid (X_{\bullet} \oplus Y_{\bullet}) = k$, the map  
$\hm_{D^b(\T \cap \L)}(X_{\bullet},Y_{\bullet}) \to \hm_{D^b_{\L}(\T)}
(X_{\bullet},Y_{\bullet})$ induced by $\iota$ is surjective.
\end{enumerate}

Let $k = 0$. We have $\supph(P_{\bullet}) \subseteq \{ m \}$ for some $m$. By Lemma \ref{change-proj}(\ref{statement4}), $P_\bullet$ is isomorphic in $D^b_{\L}(\T)$ to a shift of the module $H_m(P_\bullet)$. Since $\T$ has the $2$-out-of-$3$ property (by Remark \ref{2-3-rem}), $H_m(P_\bullet)$ is in $\T$.  Therefore, since $H_m(P_\bullet)\in\L$, we have $H_m(P_\bullet)\in\T\cap\L$, proving statement (1) for $k=0$.  We now prove (2).  Since $k=0$, $\supph(X_\bullet),\supph(Y_\bullet)\subseteq\{m\}$.  Hence, by Lemma \ref{change-proj}(\ref{statement4}), $X_\bullet$ and $Y_\bullet$ are isomorphic in $D^b(\T\cap\L)$ to $T^m H_m(X_\bullet)$ and $T^m H_m(X_\bullet)$ where $T$ is the shift functor.  Statement (2) now follows from Lemma  \ref{hom-set-same}.

Now suppose $k > 0$ and statements (1) and (2) are true for all $k'< k$.  Let $\min(P_{\bullet}) = m$. By Theorem \ref{kos-lemma2}, there exists a complex $K_\bullet$ in $\ch_{\L}(\T)$ and a chain map
$K_{\bullet} \xrightarrow{\beta} P_{\bullet}$  such that
$\supph(K_{\bullet}) = \{ m \}$, $\min_c(K_\bullet)=m$ and $H_m(\beta)$ is surjective. 
Note by Lemma \ref{change-proj}(\ref{statement4})  $K_{\bullet}$ and $T^mH_m(K_{\bullet})$ are isomorphic
in  $D^b_\L(\T)$. 
Extend $\beta$ to an exact triangle 
\[
T^{-1}C_{\bullet} \xrightarrow{\alpha} K_{\bullet} \xrightarrow{\beta} P_{\bullet} \xrightarrow{\gamma} C_{\bullet} \quad .
\]
By Lemma \ref{triang-ineq2}, we have      $\wid(T^{-1}C_{\bullet} \oplus T^mH_m(K_\bullet))
= \wid(T^{-1}C_{\bullet} \oplus K_\bullet) <k$ and
$\wid(C_{\bullet}) < k$. 
Therefore, by induction, there exists $\tilde{C}_{\bullet} \in D^b(\T \cap \L)$
such that $\iota(\tilde{C}_{\bullet}) \cong C_{\bullet}$, and the following map is a
surjection
{\scalefont{0.8}
\[
\hm_{D^b(\T \cap \L)}(T^{-1}\tilde{C}_{\bullet},T^mH_m(K_{\bullet})) \twoheadrightarrow
\hm_{D^b_{\L}(\T)}(\iota(T^{-1}\tilde{C}_{\bullet}),\iota(T^mH_m(K_{\bullet}))) \cong
\hm_{D^b_{\L}(\T)}(T^{-1}\tilde{C}_{\bullet},K_{\bullet})
\]
}
Because of this surjection, 
 there exists 
$\tilde{\alpha}:T^{-1}\tilde{C}_{\bullet} \to T^mH_m(K_{\bullet})$
in $D^b(\T \cap \L)$ with cone $\tilde{P_{\bullet}} \in D^b(\T \cap \L)$ and maps
$\tilde{\beta}$ and $\tilde{\gamma}$ so that the following diagram commutes
\begin{equation}\label{5.5}
\begin{CD}
T^{-1}C_{\bullet}					@>\alpha>>  				K_{\bullet}				@>\beta>>			 P_{\bullet}					@>\gamma>>				C_{\bullet}	\\
@VV\wr V												@VV\wr V									@.												@VV\wr V	\\
\iota(T^{-1}\tilde{C}_{\bullet})		@>\iota(\tilde{\alpha})>>  	\iota(H_m(K_{\bullet}))	@>\iota(\tilde{\beta})>>	 \iota(\tilde{P_{\bullet}})		@>\iota(\tilde{\gamma})>>		\iota(\tilde{C_{\bullet}}) \\
\end{CD}
\end{equation}
It follows that there is an isomorphism $\iota(\tilde{P_{\bullet}}) \cong P_{\bullet}$, proving statement (1).

We now prove statement (2). Let $X_{\bullet},Y_{\bullet} \in D^b(\T \cap \L)$
and $f \in \hm_{D^b_{\L}(\T)}(X_{\bullet},Y_{\bullet})$.
Set $m=\min(X_{\bullet} \oplus Y_{\bullet})$. By Lemma \ref{change-proj}(\ref{statement1}), we may assume that $\min_c(X_\bullet),\min_c(Y_\bullet)\geq m$.  From Theorem \ref{kos-thm}, we have
chain complex maps $M^X_{\bullet} \xrightarrow{\beta^X} X_{\bullet}$,
$M^Y_{\bullet} \xrightarrow{\beta^Y} Y_{\bullet}$ and $M^X_{\bullet} \xrightarrow{\kappa} M^Y_{\bullet}$ such that
\begin{itemize}
\item $M^X_{\bullet}, M^Y_{\bullet} \in \ch^b(\T \cap \L)$ are concentrated in degree $m$
\item $H_m(\beta^X)$ and $H_m(\beta^Y)$ are surjective.
\item there is a commutative diagram in $D^b_{\L}(\T)$ :
\[
\xymatrix{
M^X_{\bullet} \ar[r]^{\beta^X} \ar[d]^{\kappa} & X_{\bullet} \ar[d]^f \\
M^Y_{\bullet} \ar[r]^{\beta^Y} & Y_{\bullet} \\
}
\]
\end{itemize}

Taking cones $C^X_{\bullet}, C^Y_{\bullet}$ of $\beta^X, \beta^Y$ respectively
in $D^b(\T \cap \L)$, we get a morphism of triangles in $D^b_{\L}(\T)$ (as mentioned earlier,
 we drop the $\iota$) :
\begin{equation}\label{5.7}
\begin{CD}
T^{-1} C^X_{\bullet}				@>\alpha^X>>  				M^X_{\bullet}				@>\beta^X>>			 X_{\bullet}					@>\gamma^X>>				C^X_{\bullet}	\\
@VVT^{-1}\lambda V										@VV\kappa V									@VV fV										@VV\lambda V	\\
T^{-1} C^Y_{\bullet}					@>\alpha^Y>>  				M^Y_{\bullet}				@>\beta^Y>>			 Y_{\bullet}					@>\gamma^Y>>				C^Y_{\bullet}	\\
\end{CD}
\end{equation}
We emphasize again that all the objects are in $D^b(\T \cap \L)$ and the 
horizontal maps and $\kappa$ are morphisms from $D^b(\T \cap \L)$.
By Lemma \ref{triang-ineq}, $\wid(C^X_{\bullet} \oplus C^Y_{\bullet}) < k$, and
so, by the induction hypothesis, there exists 
$\widetilde{ \lambda}\in\hm_{D^b(\T\cap\L)}(C^X_{\bullet}, C^Y_{\bullet})$
such that $\iota(\widetilde{\lambda}) \simeq \lambda$. We now have a diagram in $D^b(\T\cap\L)$ whose 
rows are exact triangles:
\begin{equation}\label{5.8}
\begin{CD}
{T^{-1}C^X_{\bullet}}		@>\alpha^X>>  	M^X_{\bullet}
	@>\beta^X>>		X_{\bullet}		@>{\gamma^X}>>		{C^X_{\bullet}}\\
@VV T^{-1}\widetilde{\lambda} V							@VV\kappa V						@.										@VV\tilde{\lambda} V	\\
{T^{-1}C^Y_{\bullet}}		@>{\alpha^Y}>>  	M^Y_{\bullet}
@>\beta^Y>>	 	Y_{\bullet}		@>{\gamma^Y}>>	{C^Y_{\bullet}}\\
\end{CD}
\end{equation}
A priori, the left square in diagram (\ref{5.8}) may not commute. Note however that
$$\iota({\alpha^Y} \circ T^{-1}\widetilde{\lambda} - \kappa \circ {\alpha^X}) =
\iota({\alpha^Y} \circ T^{-1}\widetilde{\lambda}) - 
\iota(\kappa \circ {\alpha^X}) = \alpha^Y \circ T^{-1}\lambda -\kappa \circ \alpha^X
= 0.$$
By Lemma \ref{sp-inj}, the map  $\hm_{D^b(\T \cap \L)}(T^{-1}C^X_{\bullet}, M^Y_{\bullet}) \rightarrow \hm_{D^b_{\L}(\T)}
(T^{-1}C^X_{\bullet}, M^Y_{\bullet})$ induced by $\iota$ is injective and so $ {\alpha^Y} \circ T^{-1}\widetilde{\lambda} - \kappa \circ {\alpha^X}=0$. Hence the left square in
diagram (\ref{5.8}) commutes. Thus there exists
$X_{\bullet} \xrightarrow{g} Y_{\bullet}$ which gives a morphism of triangles
in $D^b(\T \cap \L)$ :
\[
\xymatrix{
T^{-1} C^X_{\bullet} \ar[r]^{\alpha^X} \ar[d]^{T^{-1}\widetilde{\lambda}} & M^X_{\bullet} \ar[r]^{\beta^X} \ar[d]^{\kappa} &  X_{\bullet} \ar[r]^{\gamma^X} \ar@{-->}[d]^g & C^X_{\bullet} \ar[d]^{\tilde{\lambda}}	\\
T^{-1} C^Y_{\bullet} \ar[r]^{\alpha^Y} & M^Y_{\bullet} \ar[r]^{\beta^Y} &
 Y_{\bullet} \ar[r]^{\gamma^Y} & C^Y_{\bullet}	\\
}
\]

This means if we replace $f$ with $\iota(g)$ in diagram (\ref{5.7}), the diagram
remains commutative.  Unfortunately, it does not follow from the third axiom of triangulated categories that $f = \iota(g)$.

Set $\delta = f - \iota(g)$.  We now have a morphism of triangles in $D^b_{\L}(\T)$.
\[
\xymatrix{
T^{-1} C^X_{\bullet} \ar[r]^{\alpha^X} \ar[d]^0 & M^X_{\bullet} \ar[r]^{\beta^X} \ar[d]^0 &
 X_{\bullet} \ar[r]^{\gamma^X} \ar[d]^{\delta} & C^X_{\bullet} \ar[d]^0	\\
T^{-1} C^Y_{\bullet} \ar[r]^{\alpha^Y} & M^Y_{\bullet} \ar[r]^{\beta^Y} &
 Y_{\bullet} \ar[r]^{\gamma^Y} & C^Y_{\bullet}	\\
}
\]
Hence, there exists $X_{\bullet} \xrightarrow{u} K^Y_{\bullet}$ such that 
$\delta = \beta^Y \circ u$.  Similarly, there exists a 
$C^X_{\bullet} \xrightarrow{v} Y_{\bullet}$ such that $\delta = v \circ \gamma^X$.
Hence, we have 
\[
\xymatrix{
T^{-1} C^X_{\bullet} \ar[r]^{\alpha^X} \ar[d]^0 & M^X_{\bullet} \ar[r]^{\beta^X} \ar[d]^0 &
 X_{\bullet} \ar[r]^{\gamma^X} \ar[d]^{\delta} \ar@{-->}[ld]^{u}
  & C^X_{\bullet} \ar[d]^0 \ar@{-->}[ld]^{v} \\
T^{-1} C^Y_{\bullet} \ar[r]_{\alpha^Y} & M^Y_{\bullet} \ar[r]_{\beta^Y} &
 Y_{\bullet} \ar[r]_{\gamma^Y} & C^Y_{\bullet}	\\
}
\]

%
We break the proof into three cases.
\begin{enumerate}
\item[Case (i)] If $\min(X_{\bullet}) > m$, then $\min(X_\bullet)>m=\max(M^Y_\bullet)$, and so Lemma \ref{hom-set-0} tells us that $\hm_{D^b_{\L}(\T)}(X_\bullet,M^Y_\bullet)=0$.  It follows that
$u = 0$. Hence, we have $\delta = \beta^Y \circ u = 0$. So $f = \iota(g) \in D^b(\T \cap \L)$. To summarize,
we have shown that whenever  
$\min(X_{\bullet}) > \min(Y_{\bullet}) = m$ and $\wid(X_{\bullet} \oplus Y_{\bullet}) \leq k$, we have
$\hm_{D^b(\T \cap \L)}(X_{\bullet}, Y_{\bullet}) \twoheadrightarrow 
\hm_{D^b_{\L}(\T)}(X_{\bullet}, Y_{\bullet})$.
\item[Case (ii)] If $\min(X_{\bullet}) = \min(Y_{\bullet}) = m$, then by Lemma \ref{triang-ineq}
$\wid(C^X_{\bullet} \oplus Y_{\bullet}) \leq k$ and 
$\min(C^X_{\bullet}) \geq m + 1 > m = \min(Y_{\bullet})$.  Thus, $C^X_{\bullet}, Y_{\bullet}$
satisfy the hypothesis of the already proved case(i) and so  we have $\hm_{D^b(\T \cap \L)}(C^X_{\bullet}, Y_{\bullet}) \twoheadrightarrow 
\hm_{D^b_{\L}(\T)}(C^X_{\bullet}, Y_{\bullet})$. Hence we have
$v \in D^b(\T \cap \L)$, and hence $\delta = v \circ \gamma^X \in D^b(\T \cap \L)$. Therefore, $f = \iota(g) + \delta$ is in $D^b(\T \cap \L)$.
\item[Case (iii)]If $\min(X_{\bullet}) = m < \min(Y_{\bullet})$, then
$\wid(C^X_{\bullet} \oplus Y_{\bullet}) \leq k-1$ by Lemma \ref{triang-ineq}. Therefore $v \in D^b(\T \cap \L)$ by the induction
hypothesis, and, as in case(ii), $f \in D^b(\T \cap \L)$.
\end{enumerate}
This finishes the proof.
\end{proof}

\begin{thm}\label{thick-main}
The functor $\iota:D^b(\T \cap \L) \to D^b_{\L}(\T)$ is an equivalence.
\end{thm}
\begin{proof}
By Proposition \ref{ess-surj-full}, $\iota$ is full and essentially surjective.  Furthermore, $\iota$ preserves homologies, thus $\iota(X_\bullet)$ is acyclic if and only if $X_\bullet$ is acyclic.  Therefore, $\iota$ is faithful on objects.  The equivalence then follows from Lemma \ref{FinalEquivalence}.
\end{proof}

We now come to the case when $\T = \bA$ where $\A$ is a resolving subcategory of
$\M(R)$.
\begin{lemma}\label{resolv-main}
There is an equivalence of categories induced by chain complex functors
$\iota' : D^b_{\L}(\A) \stackrel{\sim}{\rightarrow} D^b_{\L}(\bA)$.
\end{lemma}
\begin{proof}
The proof of essential surjectivity is similar to the proof in Lemma \ref{change-proj}(3). 
The proof that $\iota'$ fully faithful
is easy to obtain and standard.
\end{proof}

As a straightforward consequence of Lemma \ref{resolv-main} and Theorem \ref{thick-main}
when $\T = \bA$, we can now obtain the main result.
\begin{thm}\label{main}
There is an equivalence of categories
$D^b(\bA \cap \L) \simeq D^b_{\L}(\A)$.
\end{thm}
The following interesting corollary can now be obtained from the above result.
\begin{cor}\label{cor-main}
Suppose $R$ is coequidimensional and Cohen-Macaulay, then we have the equivalence
$$D^b(\FP \cap \Mfl) \simeq D^b_{fl}(\P(R)).$$
\end{cor}
\begin{proof}
When $R$ is Cohen-Macaulay and coequidimensional, the category $\L = \Mfl$ is a Serre subcategory satisfying condition (*). Hence, the equivalence follows from Theorem \ref{main} by taking $\L = \Mfl$ and $\A = \P(R)$.
\end{proof}
Furthermore, our result holds for any resolving subcategory in Example \ref{res-ex}
and Serre subcategory in Example \ref{serre-ex}. We explicitly state some important
cases. In the special case when $R$ is Cohen-Macaulay and equicodimensional, $\A = \M(R)$ and 
$\L = \Mfl$, we obtain the well-known equivalence
used in most d\'{e}vissage statements (refer \cite[1.15(Lemma,Ex. (b))]{K}).
\begin{cor}
$$D^b(\Mfl) \simeq D^b_{fl}(\M(R)).$$
\end{cor}
Note that this equivalence is known even without the assumption that $R$ is Cohen-Macaulay.
\begin{cor}
Let $R$ be Cohen-Macaulay.
Let $V$ be any set theoretic complete intersection in $\spec(R)$ and $c$ be any integer.
Let $\FP^c_V$ denote the category of modules with finite projective dimension supported
on $V$ and in codimension at least $c$. Let $D^c_V(\P(R))$ denote the derived category
with chain complexes of projective modules with homologies supported on $V$ and in
codimension at least $c$.
Then $$ D^b(\FP^c_V) \simeq D^c_V(\P(R)).$$
\end{cor}
Note that without $c$, the above result holds even without $R$ being Cohen-Macaulay.

The main theorem, Theorem \ref{main} is also related to an interesting corollary
of the oft-quoted Hopkins-Neeman theorem \cite{H,N2} for perfect complexes. 
Let $\L$ be any Serre subcategory of $\M(R)$. A consequence of the Hopkins-Neeman theorem 
is that $thick_{D^b(\M(R))}(\xbar{\P(R)} \cap \L) \simeq D^b_{\L}(\P(R))$ where $thick$
is the thick closure (note that here we use thick in the triangulated sense). We
generalize this as follows.
\begin{cor} \label{hop-nee}
Let $\L$ be a Serre subcategory satisfying condition (*). Let $\T$ be any thick subcategory
of $\M(R)$. Then the thick closure (in the triangulated sense) of $\T \cap \L$ in $D^b(\M(R))$
is $D^b_{\L}(\T)$ (after completion with respect to isomorphisms).
\end{cor}
\begin{proof}
Note that there is a commutative square 
\[
\xymatrix{
K^+_{\T}(\P(R)) \ar[r] \ar[d]^{\wr} & K^+(\P(R)) \ar[d]^{\wr} \\
D^+(\T) \ar[r] & D^+(\M(R)).\\
}
\]
The top horizontal arrow is a full embedding, hence so is the bottom. Hence, restricting
to the bounded category, $D^b(\T)$ is a thick subcategory of $D^b(\M(R))$ and hence so is $D^b_{\L}(\T)$ (after completing them with respect to isomorphisms). However, up to completion with respect to isomorphisms, 
$$\text{image}(D^b(\T \cap \L)) \subseteq thick_{D^b(\M(R))}(\T \cap \L) \subseteq D^b_{\L}(\T).$$ But by Theorem \ref{thick-main}, we then get the required result.
\end{proof}

Finally we prove the theorem we promised at the beginning of the article.

\begin{thm}\label{thm-cm-iff}
Suppose $R$  is local.  Then $R$ is Cohen-Macaulay if and only if the following equivalence holds 
$$D^b(\FP \cap \Mfl) \simeq D^b_{fl}(\P(R)).$$
\end{thm}
\begin{proof}
When $R$ is Cohen-Macaulay, the equivalence follows from Theorem \ref{main} by taking $\L = \Mfl$ and $\A = \P(R)$.  

Now suppose $R$ is not Cohen-Macaulay. The new intersection theorem, \cite{R}, asserts that such a ring $R$ never admits a finite length, finite projective dimension module.
Thus $D^b(\FP \cap \Mfl) = 0$. However, the Hopkins-Neeman theorem  \cite{H,N2} states that the thick subcategories of $D^b(\P(R))$ are in bijective correspondence with specialization closed subsets of $\spec(R)$ and that this bijection is given by taking the supports of objects in $D^b(\P(R))$.  Letting $m$ be the maximal ideal of $R$, since the specialization closed set $\{m\}$ differs from the specialization closed subset $\emptyset$,  the Hopkins-Neeman theorem implies that the subcategory of $D^b(\P(R))$ supported on $\{m\}$, which is $D^b_{fl}(\FP)$, differs from that supported on $\emptyset$, which is $0$.  Hence $D^b_{fl}(\FP)$ cannot be $0$ and so the equivalence fails.
\end{proof}

\section{Results on Homological Functors}
 \label{sec5}
Now that we have proved the equivalence of the two categories, we can compare various theories of invariants for derived equivalences. We restrict our attention to $\K$-theory and triangular Witt and Grothendieck-Witt groups. 

\section*{$\K$-theory comparisons and results}
Since $K$-theoretic invariants need not always be preserved by equivalences of derived
categories (\cite{S1}), we will need to view the categories above with some more structure.
While the original and several other articles (\cite{WF},\cite{TT},\cite{N1},\cite{B1},\cite{B2},\cite{WC}, \cite{S2},\cite{S3}) serve as good references for this part, we will refer to the articles (\cite{S4}, \cite{T}) for the terminology and results.

The categories $\ch^b_{\L}(\A), \ch^b(\T \cap \L)$ and $\ch^b_{\L}(\T)$ are strongly
pretriangulated $dg$-categories and the natural functors are functors of such categories.
In particular, with the usual choices of weak equivalences as quasi-isomorphisms, these
are all complicial exact categories with weak equivalences, and the natural functors
preserve weak equivalences. Assume, as usual, that $\L$ also satisfies condition (*).

Let $\K$ be the nonconnective K-theory spectrum. Applying Theorem \ref{thick-main} and 
\cite[Theorem 3.2.29]{S4}, we get that $\iota$ induces homotopy equivalences of
$\K$-theory spectra. Similarly, applying Lemma \ref{resolv-main} and \cite[Theorem 3.2.29]{S4},
 we get that $\iota'$ induces homotopy equivalences of $\K$-theory spectra.

Putting these together and further using \cite[3.2.30]{S4}, we obtain
\begin{thm}\label{k-th}
The spectra $\K(\bA \cap \L)$ and $\K(D^b_{\L}(\A))$ are homotopy equivalent. Hence
$$\K_i(\bA \cap \L) \simeq \K_i(D^b_{\L}(\A)) \quad \forall i \in \Z.$$
\end{thm}
Once again, this result holds for any resolving subcategory in Example \ref{res-ex} and
Serre subcategory in Example \ref{serre-ex}. We list the most important corollary.
\begin{cor}\label{k-thp}
Let $R$ be Cohen-Macaulay and equicodimensional. Then there is a homotopy equivalence between
$\K(\FP \cap \Mfl)$ and $\K(D^b_{fl}(\P(R)))$. Hence
$$\K_i(\FP \cap \Mfl) \simeq \K_i(D^b_{fl}(\P(R)))$$
\end{cor}
As mentioned in Section \ref{sec1}, special cases of Corollary \ref{k-thp} were known earlier.

When $R$ is Cohen-Macaulay and equicodimensional of dimension $d$, the special case of $\A = \M(R)$ and
$c = d$ in Theorem \ref{k-th} gives us the well-known equivalence for coherent $\K$-theory.

\begin{rem}\label{spectral-kth-rmk}
Let $X$ be a (topologically) noetherian scheme with a bounded generalized dimension
function as in \cite{B3} . Then coniveau and niveau spectral sequences are defined in
\cite[Theorem 1, Theorem 2]{B3} converging to the $\K$-groups of $X$. The $q^{th}$
row on the $E_1$ page consists of unaugmented Gersten-like complexes with terms 
$$\bigoplus_{x \in X^{(p)}} \K_{-p-q}(\O_{X,x} \text{ on } {x}) \qquad \text{and} \qquad
\bigoplus_{x \in X_{(-p)}} \K_{-p-q}(\O_{X,x} \text{ on } {x})$$ respectively.
Further, there is an augmented weak Gersten complex for the usual codimension and dimension
functions, as defined in \cite{B3} . When $X$ is regular, Quillen's d\'{e}vissage
theorem can be applied to rewrite these terms as the K-theories of the residue fields of the points. However, when $X$ is not regular, the theorem does not apply and hence
the terms remain as abstract K-groups of derived categories supported at the points.

Now under the further assumption that all the local rings $\O_{X,x}$ are Cohen-Macaulay
(i.e. $X$ is Cohen-Macaulay), we can apply Theorem \ref{k-thp} and rewrite these spectral sequences
in terms of the $\K$-groups of the category of finite length, finite projective dimension
modules over the local rings at the points. Thus the computation of global $\K$-groups
can now be reduced to computing $\K$-groups of the category $\FP \cap \Mfl(R)$ for a 
Cohen-Macaulay local ring $R$.
\end{rem}
To summarize, we obtain :
\begin{thm}\label{spectral-kth-thm}
For a Cohen-Macaulay scheme $X$ of dimension $d$, we have spectral sequences :
$$E_1^{p,q} = \bigoplus_{x \in X^{(p)}} \K_{-p-q} (\xbar{\P(\O_{X,x})} \cap \Mfl(\O_{X,x})) 
\stackrel{p+q=n}{\Longrightarrow} \K_{-n}(X)    \qquad \text{ and }$$
$$E_1^{p,q} = \bigoplus_{x \in X_{(p)}} \K_{-p-q} (\xbar{\P(\O_{X,x})} \cap \Mfl(\O_{X,x}))
\stackrel{p+q=n}{\Longrightarrow} \K_{-n}(X)$$
and augmented weak Gersten complexes for each $q \in \Z$ :
\begin{align*}
& \K_q(X) \rightarrow \bigoplus_{x \in X^{(0)}} \K_q(\xbar{P(\O_{X,x})} \cap \Mfl(\O_{X,x})) \rightarrow \\
& \qquad \bigoplus_{x \in X^{(1)}} \K_{q-1}(\xbar{\P(\O_{X,x})} \cap \Mfl(\O_{X,x}))
\rightarrow \ldots \rightarrow
\bigoplus_{x \in X^{(d)}} \K_{q-d}(\xbar{\P(\O_{X,x})} \cap \Mfl(\O_{X,x})).\\
\end{align*}
\end{thm}

\section*{Witt and Grothendieck-Witt group comparisons and results}
Let $R$ be an equicodimensional Cohen-Macaulay ring of dimension $d$. We now consider
the situation where the category $\A$ is a duality-closed thick subcategory of $\G_C$,
the category arising from a semidualizing module $C$ as defined in Section \ref{semidual}
with duality given by $\hm(\_ ,C)$. Let $\L = \Mfl$. We assume further that $2$ is
invertible in the ring $R$.

There is a duality on the category $\bA \cap \L$ given by $\ext_R^d(\_ ,C)$ which
induces a duality on $D^b(\bA \cap \L)$. Similarly, there is a duality on
$D^b_{\L}(\A)$ given by $\hm_R(\_ ,C)$ (or $\da$) respectively.

The arguments in either of \cite[Lemma 6.4]{BW}, \cite{Gi} or \cite{MS} go through
with minimal modifications showing that these are indeed dualities and that they
are preserved by the "resolution functor", the composite functor $\iota \circ \iota'^{-1}$
from $D^b(\bA \cap \L)$ to $D^b_{\L}(\A)$. 

A direct application of \cite[Lemma 4.1(c)]{BW} 
now yields
\begin{thm}\label{witt}
There is an isomorphism of triangular Witt groups
$$W(\bA \cap \L) \stackrel{\sim}{\rightarrow} W^0(D^b(\bA \cap \L)) \stackrel{\sim}{\rightarrow} W^d(D^b_{\L}(\A)).$$
\end{thm}
The case $\A = \P(R)$ results in the following new corollary.
\begin{cor}\label{wittp}
When $R$ is equicodimensional, there is an isomorphism of triangular Witt groups
$$W(\FP \cap \Mfl) \stackrel{\sim}{\rightarrow} W^d(D^b_{fl}(\P(R)))$$
\end{cor}
The special case of $R$ being equicodimensional and $C = D$ a dualizing module (or complex)
gives us the result in the well-known coherent case \cite{Gi} .

\begin{rem} \label{gw}
Similarly, a direct application of \cite[Theorem 2.1]{WC} yields the same results for
Grothendieck-Witt groups as for Witt groups above.
\end{rem}

In \cite{MS}, the authors define a new Witt group for exact subcategories of triangulated categories closed under duality. They further prove another form of d\'{e}vissage for
triangular Witt groups, namely
\[
\xymatrixcolsep{3pc}\xymatrix{
W(\FP \cap \Mfl) \ar[r]^-{\sim}_-{\alpha} &  W(D^b_{\sFP \cap \Mfl}(\FP \cap \Mfl)) \ar[r]^-{\sim}_-{\alpha'} \ar[d]^-{\wr}_-{\beta} & W(D^b(\FP \cap \Mfl)) \\
& W^d(D^b_{fl}(\P(R)))) & \\
}
\]
With our notations as above, we can now generalize and improve upon this picture to obtain
the following theorem.
\begin{thm}\label{witt-comp}
There are natural isomorphisms of Witt groups :
\[
\xymatrixcolsep{3pc}\xymatrix{
W(\bA \cap \L) \ar[r]^-{\sim}_-{\alpha} &  W(D^b_{\sbA \cap \L}(\bA \cap \L)) \ar[r]^-{\sim}_-{\alpha'}
\ar[d]^-{\wr}_-{\beta} & W(D^b_{\L}(\bA \cap \L)) \ar[d]^-{\wr}_-{\gamma} \\
& W^d(D^b_{\sbA \cap \L}(\A)) \ar[r]^-{\sim}_-{\delta} & W^d(D^b_{\L}(\A))\\
}
\]
\end{thm}
\begin{proof}
We only give a sketch of the proof. The isomorphism $\gamma$ already occurs in Theorem
\ref{witt} and $\beta$ is also a similar direct consequence of Theorem \ref{main} by
restricting the equivalence of categories  to the categories with support $\bA \cap \L$
(or by following the arguments in \cite{MS}). The arguments in \cite{MS} generalize
directly to give the isomorphisms $\alpha$ and $\alpha'$. The commutativity of the
diagram shows that $\delta$ is also an isomorphism.
\end{proof}

\begin{rem}
The remark \ref{spectral-kth-rmk} and succeeding theorem \ref{spectral-kth-thm}
works (as in \cite[Remark 3]{B3}) for any cohomology theory which induces long
exact sequences on short exact sequences of triangulated categories. In particular,
for Witt theory, we would get spectral sequences and augmented weak Gersten-Witt complexes
for triangular Witt groups tensored with $\Z[\frac{1}{2}]$ as in Theorem \ref{spectral-kth-thm} .
\end{rem}

\section{Examples and Questions}\label{sec6}
An advantage of working with arbitrary Serre subcategories is the ability to deal with supports. We are able to deal with supports only when the Serre subcategory satisfies the condition (*). In contrast, in the coherent picture (i.e. $G$ theory or coherent Witt groups), theorems similar to the ones in section 5 exist with arbitrary supports, i.e. supports in any specialization closed
subset, in particular over any closed subset $V$ of $\spec(R)$. This is one reason why
smoothness has played a crucial role in results for $\K$-theory or other similar theories, since both coherent and usual theories coincide. Thus the following question is natural :
\begin{quest}
Is $D^b(\T \cap \L) \rightarrow D^b_{\L}(\T)$ an equivalence for any thick
subcategory $\T$ of $\M(R)$ and any Serre subcategory $\L$ of $\M(R)$?
\end{quest}
As we noted in Section \ref{sec1}, when $R$ is local and not Cohen-Macaulay, this is always false
with $\L = \Mfl$ and $\T = \xbar{\P(R)}$. However, it is still plausible that the result holds for a
more general class of Serre subcategories.

Next, we consider a question about quotients. Supports allow one to write localization
exact sequences. Comparing supports with quotients is a powerful tool, for example it is known that
$\displaystyle{\K_i\left(\M\left(\frac{R}{(a)}\right)\right) \simeq \K_i(\{ \text{modules supported on V(a)} \})}$
where $a$ is a nonzero divisor. This is because one can either apply d\'{e}vissage directly,
or, for other similar theories, the spectral sequence and d\'{e}vissage reduces one to the
case of residue fields of points and both sides have the same residue fields. This leads to the following question :
\begin{quest}
Let $\L$ be the Serre subcategory of modules supported on $V(I)$ where $I$ is a set
theoretic complete intersection ideal. Let $\T$ be a thick subcategory in $\M(R)$.
Is $\K_i\left(D^b\left(\T \cap \M\left(\frac{R}{I}\right)\right)\right) \rightarrow \K_i\left(D^b_{\L}(\T)\right)$ an isomorphism?
\end{quest}
We present a rather simple example which answers this question negatively.
\begin{exem}
Let $\displaystyle{R = \frac{k[X]}{(X^2)} }$. Let $I = (X)$, so $V(I) = \spec(R)$.
Note that $V(I) = V(\emptyset)$ and so $V(I)$ is a set theoretic complete intersection.
Let $\T = \xbar{\P(R)}$. Then $\T \cap \M(R/I) = \{ 0 \}$ while 
$D^b_{\L}(\T) = D^b(\xbar{\P(R)}) \neq \{ 0 \}$.
Then $K_0(\T \cap \M(R/I)) = 0$ and $K_0(D^b_{\L}(\T)) = \Z$.
\end{exem}

Clearly going modulo any ideal $I$ does not work. We specialize to the case when
$I = (a)$ where $a$ is a nonzero divisor.
\begin{quest}\label{open}
Let $a$ be a nonzero divisor. Let $\L$ be the Serre subcategory of modules supported
on $V(a)$. Is $\K_i \left(\bA \cap \M\left(\frac{R}{(a)}\right)\right) \xrightarrow{} \K_i(\bA \cap \L)$ an
isomorphism?
\end{quest}
One natural way to answer this question would involve the following two steps :
\begin{enumerate}
\item Reprove Quillen's d\'{e}vissage theorem for full subcategories of $\M(R)$ satisfying the 2-out-of-3 property
(it is known for abelian subcategories).
\item Find a natural filtration for a module $M \in \bA \cap \L$ so that the quotients belong to $\bA \cap \M\left(\frac{R}{(a)}\right)$.
\end{enumerate}
For every module $M$ on the right there exists an $n$
such that $a^nM = 0$ and thus a natural filtration : 
$M \supseteq aM \ldots \supset a^{n-1}M \supset 0$
and $\displaystyle{\frac{a^iM}{a^{i+1}M}}$ is in $\M(\frac{R}{(a)})$. At first sight, this might seem like an answer to the second part. However, it turns out that even though $M \in \A$, its quotient $\displaystyle{\frac{M}{aM}}$ need
not be. Hailong Dao provided us with the following example in the best possible case of a polynomial variable.

\begin{exem}
Let $\displaystyle{R = \frac{k[[X,Y]]}{(XY)}[Z]}$. Let 
$\displaystyle{M = \frac{R}{(X-Z, Y -Z)}}$. Then $M$ has finite projective dimension
over $R$, $Z^2 M = 0$ but $\displaystyle{\frac{M}{(Z)M} = k}$ does not have finite
projective dimension over $\displaystyle{\frac{k[[X,Y]]}{(XY)}}$.
\end{exem}
Since the module $M$ has length $2$, there is no other option of a filtration. Thus, the most
natural naive arguments provide no answer. If Question \ref{open} has a positive
answer, it would yield nice long exact sequences and be useful in computations.

Finally, let $R$ be equicodimensional and Cohen-Macaulay of dimension $d$, $\A \subseteq \G_C$ a thick subcategory closed under 
the duality, and $\L$ the Serre subcategory of finite length modules supported on a set 
theoretic complete intersection $V$. Since we have an equivalence of triangular Witt groups 
and triangular Grothendieck-Witt groups (Theorem \ref{witt}, Remark \ref{gw}) and both are obtained from a  Grothendieck-Witt spectrum, it begs the natural question :
\begin{quest}
Are the Grothendieck-Witt spectra of $D^b(\bA \cap \L)$ and $D^b_{\L}(\A)$ homotopy equivalent?
\end{quest}
If a suitable intermediate category can be found which has duality, then we can answer this
in the affirmative. However, there is in general no duality on $\bA$. Also, we do not know
if the category of double complexes constructed in the proof of \cite[Lemma 6.4]{BW} arises
as the homotopy category of some suitable pretriangulated category.

\end{document}